\documentclass[reqno,10pt,centertags,draft]{amsart}
\usepackage{amsmath,amsthm,amscd,amssymb,latexsym,upref}
\date{\today}

\newcommand{\bbD}{{\mathbb{D}}}

\newcommand{\bbR}{{\mathbb{R}}}

\newcommand{\bbZ}{{\mathbb{Z}}}

\newcommand{\bbT}{{\mathbb{T}}}

\newcommand{\cG}{{\mathcal{G}}}
\newcommand{\cR}{{\mathcal{R}}}

\renewcommand{\a}{\alpha}
\newcommand{\be}{\beta}
\newcommand{\ga}{\gamma}

\newcommand{\de}{\delta}
\newcommand{\ep}{\varepsilon}
\newcommand{\s}{\sigma}

\newcommand{\la}{\lambda}
\renewcommand{\s}{\sigma}
\newcommand{\om}{\omega}
\newcommand{\z}{\zeta}

\newcommand{\Bes}{B_2^{1/2}}
\renewcommand{\(}{\left(}
\renewcommand{\)}{\right)}
\newcommand{\wh}{\widehat}

\allowdisplaybreaks
\numberwithin{equation}{section}
\newtheorem{theorem}{Theorem}[section]

\newtheorem{proposition}[theorem]{Proposition}

\theoremstyle{definition}
\newtheorem{definition}[theorem]{Definition}
\newtheorem{remark}[theorem]{Remark}

\newtheorem{example}[theorem]{Example}

\begin{document}

\title[]{The scattering problem in Ryckman's class of Jacobi matrices}
\author[]{L. Golinskii, A. Kheifets$^*$, P. Yuditskii$^{**}$}

\thanks{$^*$The work was supported by
the University of Massachusetts Lowell Research and Scholarship Grant,
project number: H50090000000010}

\thanks{$^{**}$The work was supported by the Austrian Science
Found FWF, project number: P22025--N18}

\address{Institute for Low Temperature Physics and Engineering,
47 Lenin Avenue, Kharkov 61103, Ukraine}
\email{golinsky@ilt.kharkov.ua}

\address{
Department of Mathematics, University of Massachusetts Lowell,
01854, USA} \email{Alexander\_Kheifets@uml.edu}

\address{Institute for Analysis, Johannes Kepler University Linz,
A-4040 Linz, Austria} \email{franz.peherstorfer@jk.uni-linz.ac.at}

\address{Institute for Analysis, Johannes Kepler University Linz,
A-4040 Linz, Austria} \email{petro.yuditskiy@jku.at}

\date{\today}
\begin{abstract}
We give a complete solution of the scattering problem for Jacobi
matrices from a class which was recently introduced by E. Ryckman.
We characterize the scattering data for this class and illustrate
the inverse scattering on some simple examples.
\end{abstract}

\maketitle

\section{Introduction}

In mid 70th Guseinov \cite{Gus76,Gus77} developed a scattering
theory for infinite Jacobi matrices
\begin{equation}\label{jacmat}
J=J(\{a_n\},\{b_n\})=
\begin{bmatrix}
b_1&a_1&0&\ldots\\
a_1&b_2&a_2&\ldots \\
0&a_2&b_3&\ldots \\
\vdots&\vdots&\vdots&\ddots
\end{bmatrix},
\end{equation}
$a_n>0$, $b_n=\bar b_n$, which can be viewed as a discrete version
of the scattering theory for one-dimensional Schr\"odinder operator
on the half--line by Marchenko--Faddeev. The basic assumption on $J$
is the finiteness of the first moment
\begin{equation}\label{1stmom}
    \sum\limits_{n=1}^{\infty}n(|a_n-1|+|b_n|)<\infty .
\end{equation}
We say that a Jacobi matrix $J=J(\{a_n\},\{b_n\})$ belongs to
Guseinov's class $\cG$ if its parameters satisfy (\ref{1stmom}).
Later Geronimo \cite{Ger80,Ger94} (see also \cite{Gernev}) solved
the spectral problem for Jacobi matrices in more general
``weighted'' Guseinov's classes by using the inverse scattering
technique. The main feature of his results is that the decay of the
Jacobi parameters $\{a_n-1\}$, $\{b_n\}$ manifests itself in the
decay of the Fourier coefficients of the absolutely continuous part
(after suitable modifications) of the measure.

In the modern scattering theory of Jacobi operators (see, e.g.,
\cite{Egmite1}) the various analogues of \eqref{1stmom} (for much
more complex backgrounds) appeared, and they seemed to be
indispensable.

In 2007 Ryckman \cite{Ryck2,Ryck3,RyckPhD} came up with a new class
of Jacobi matrices, for which he obtained a complete spectral
description. To state his result we introduce some notations and
definitions. Let us write
$$
\be=\{\be_n\}\in\ell^2_s, \quad s>0\qquad{\rm if}\quad
\|\be\|^2_{\ell^2_s}:= \sum\limits_n |n|^s|\be_n|^2<\infty.
$$
\begin{definition}\label{001}
A Jacobi matrix $J=J(\{a_n\},\{b_n\})$ belongs to Ryckman's class
$\cR$, or its spectral measure $\s(J)\in \cR$, if the series $\sum_n
(a_n-1)$ and $\sum_n b_n$ are conditionally summable, and
$$
\xi_n:=-\sum\limits_{k=n+1}^\infty b_k\in\ell^2_1,\quad
\eta_n:=-\sum\limits_{k=n+1}^\infty (a_k-1)\in\ell^2_1
$$
\end{definition}
\begin{definition}\label{002}
A function $g$ on the unit circle $\bbT$ is said to be in the Besov
class $\Bes$ if the sequence of its Fourier coefficients is in
$\ell^2_1(\bbZ)$
\begin{equation}\label{besov}
g(t)=\sum_{n\in\bbZ} g_nt^n, \qquad \sum_{n\in\bbZ}
|n||g_n|^2<\infty.
\end{equation}
 Also, a function $f$ on the interval $[-2,2]$ is
said to be in $\Bes$ if $\wh f(t):=f\left(t+\frac{1}{t}\right)$ is
in $\Bes$.
\end{definition}
Note that $\wh f$ is a symmetric function, $\wh f(\bar t)=\wh f(t)$,
and conversely, each symmetric function has the form $\wh f$. If
$$ \wh f(t)=\sum_{n\in\bbZ} \wh f_n t^n, $$
then the symmetry of $\wh f$ implies $\wh f_{-n}=\wh f_n$. If $\wh
f$ is in addition a real function, then $\wh f_{-n}=\wh
f_n=\overline{\wh f_n}$.

\begin{theorem}[Ryckman]\label{ryckfund}
$J\in \cR$ if and only if the spectral measure $\s(J)$ of $J$ has
this structure:
\begin{itemize}
\item The absolutely continuous part is supported by $[-2,2]$, and
\begin{equation}\label{acsp}
\s_{ac}(dx)=f(x,J)dx=\frac{\rho(x,J)}{2\pi}\sqrt{4-x^2}\,dx, \quad
\rho(x,J)=\frac{\rho_0(x,J)}{(2-x)^{\ga_1}(2+x)^{\ga_2}}
\end{equation}
with $\gamma_1, \gamma_2$ equal $0$ or $1$, and $\log\rho_0\in
\Bes$.
\item The singular part is
\begin{equation}\label{sinsp}
\s_{s}(dx)=\sum_{k=1}^{N}\s_k(J)\de(\la_k), \quad N=N(J)<\infty,
\quad \s_k(J)>0, \quad \la_k\in\mathbb R \setminus [-2,2].
\end{equation}
\end{itemize}
\end{theorem}

Note that $\cG\subset\cR$, and the inclusion is proper. Indeed,
\begin{equation*}
\begin{split}
\sum_{n=1}^\infty n|\xi_n|^2 =& \sum_{n=1}^\infty
n\(\sum_{k=n+1}^\infty b_k\)^2 =\sum_{n=1}^\infty \sum_{k=n}^\infty
\sum_{l=n}^\infty n|b_k||b_l| \\
\le & \sum_{n=1}^\infty \sum_{k=n}^\infty \sum_{l=n}^\infty l|b_k|
|b_l| \le\sum_{n=1}^\infty \sum_{k=n}^\infty
\sum_{l=1}^\infty l|b_k||b_l| \\
=& \sum_{k=1}^\infty \sum_{n=1}^k \sum_{l=1}^\infty l|b_k||b_l| =
\sum_{k=1}^\infty k \sum_{l=1}^\infty l|b_k| |b_l|=
\(\sum_{k=1}^\infty k |b_k|\)^2.
\end{split}
\end{equation*}
Similarly,
$$ \sum_{n=1}^\infty n|\eta_n|^2\le \(\sum_{k=1}^\infty k |a_k-1|\)^2.
$$
On the other hand, $J(\{a_n\},\{b_n\})$ with
$$ a_n=1+(-1)^n n^{-1-\epsilon},\quad b_n=(-1)^n n^{-1-\epsilon}, \quad 0<\epsilon<1 $$
belongs to $\cR$, but
\eqref{1stmom} is false.

Given $\s(J)\in\cR$, define two outer functions which are key
ingredients in both spectral and scattering theory of the class
$\cR$. First, we put
\begin{equation}\label{outer1}
\begin{split}
D_0(z)=&
D_0(z,J):=\exp\left\{\frac12\,\int_{\bbT}\frac{t+z}{t-z}\,\log\wh\rho_0(t,J)m(dt)\right\}
\\ =&
\exp\left\{\frac{u_0(z)+iv_0(z)}2\right\}
\end{split}
\end{equation}
with
$$
u_0(z) = \int_{\bbT} P_z(t)\log\wh\rho_0(t,J)m(dt), \qquad v_0(z) =
\int_{\bbT} Q_z(t)\log\wh\rho_0(t,J)m(dt)=\widetilde u_0(z)
$$
a harmonic conjugate to $u_0$,
$$ u_0(t)=\log\wh\rho_0(t,J)=\sum_{k\in\bbZ} r_k t^k, \qquad
v_0(t)=\frac1{i}\sum_{k\in\bbZ} (\text{sgn}\ k) r_k t^k, $$ both
$u_0$ and $v_0$ are real, $u_0$ is symmetric, and $v_0$
antisymmetric: $v_0(\bar t)=-v_0(t)$. So
\begin{equation}\label{symm}
D_0(\bar t)=\overline{D_0(t)}, \quad |D_0(t)|^2=\wh\rho_0(t,J)
\end{equation}
almost everywhere on $\bbT$. It is known (cf., e.g.,
\cite[Proposition 6.1.5]{Simon1}) that $\log\wh\rho_0\in\Bes$
implies $(\wh\rho_0)^{\pm1}\in L^p(\bbT)$ for $p<\infty$, so
$D_0^{\pm1}\in H^p$ for such $p$.

Secondly, we put
\begin{equation}\label{outer2}
D(z)=D(z,J):=\exp\left\{\frac12\,\int_{\bbT}\frac{t+z}{t-z}\,\log\wh\rho(t,J)m(dt)\right\}
=\frac{D_0(z,J)}{(1-z)^{\ga_1}(1+z)^{\ga_2}}.
\end{equation}

Both $D_0$ and $D$ are related to the absolutely continuous part of
the spectral measure. The discrete part is completely determined by
the set of eigenvalues $\{\la_k\}$, or equivalently, by the set
\begin{equation}\label{eigen}
Z(J):=\left\{z_k(J): \la_k=z_k(J)+\frac1{z_k(J)}, \ \
k=1,2,\ldots,N\right\},
\end{equation}
$z_k(J)\in (-1,1)\backslash \{0\}$, and by the set of masses
$\{\s_k(J)\}_{k=1}^N$ in \eqref{sinsp}.
\begin{definition} Given $J\in\cR$, under the {\it scattering data}
for $J$ we mean the following collection $\{\ga_1(J),\ga_2(J);\
Z(J);\ \mu_1(J),\ldots,\mu_N(J);\ s(t,J)\}$
\begin{enumerate}
    \item A pair $(\ga_1(J),\ga_2(J))$ from \eqref{acsp};
    \item The set $Z(J)$ from \eqref{eigen}, or equivalently, a
    finite Blaschke product
\begin{equation}\label{bla}
B(z,J)=\prod_{k=1}^N
\frac{z_k(J)}{|z_k(J)|}\,\frac{z-z_k(J)}{1-z_k(J) z}\,,\quad B(\bar
t,J)=\overline{B(t,J)}=\frac1{B(t,J)}\,,\ \ t\in\bbT;
\end{equation}
    \item $N=N(J)$ positive numbers
\begin{equation}\label{normcon}
\mu_k(J):=\s_k(J)\left|\frac{B'(z_k,J)}{D(z_k,J)}\right|^2\,\left|1-z_k(J)^{-2}\right|^{-2}>0,
\quad k=1,2,\ldots,N;
\end{equation}
    \item The scattering function
\begin{equation}\label{scafun}
s(t,J):=\frac{\phi_0(t,J)}{\phi_0(\bar t,J)}\,,
\end{equation}
$\phi_0$ is the Jost function for $J$ (see Section 2).
\end{enumerate}
\end{definition}

Compared to \cite{Ger80,Ger94} we move backward, from spectral to
scattering. The goal of the present note is to obtain a complete
characterization of the scattering data in Ryckman's class, and so
demonstrate that the scattering theory goes far beyond Guseinov's
class \eqref{1stmom}. We analyze the scattering data, prove the
uniqueness theorem in Section 2, and solve the inverse scattering
problem in Section~3. A few simple examples are given in Section 4.

\section{Scattering data}
The basic three-term recurrence relation for a Jacobi matrix
$J(\{a_n\},\{b_n\})$
$$ a_{n-1}y_{n-1}+b_ny_n+a_ny_{n+1}=(z+z^{-1})y_n, \quad
n=1,2,\ldots, \quad a_0=1 $$ has two ``distinguished'' solutions.
The first one, known as the sine-type solution, is
$$ y_n=s_n(z)=p_{n-1}\(z+\frac1z\), \qquad s_0=0, \ \ s_1=1, $$
$p_k$ are orthonormal polynomials with respect to the spectral
measure $\s(J)$. A fundamental result by Szeg\H{o} concerns an
asymptotic behavior of orthonormal polynomials with respect to
``nice'' measures with $\text{supp}\s\subset[-2,2]$. It was extended
substantially in \cite{Gon75,Nik84} and \cite{Pehyud}, where a
finite (respectively, infinite) number of mass points outside
$[-2,2]$ is allowed. In our notation the Szeg\H{o} asymptotics
states that for $J\in\cR$
\begin{equation}\label{szego}
Q(z):=\lim_{n\to\infty}z^np_n\(z+\frac1z\)=\frac{B(z,J)}{(1-z^2)D(z,J)}
\end{equation}
uniformly on the compact subsets of the unit disk $\bbD$.

The second solution is the Jost solution $y_n=\phi_n$, defined by a
specific asymptotic behavior at infinity
$$ \lim_{n\to\infty} z^{-n}\phi_n(z,J)=1 $$
uniformly on the compact subsets of the unit disk $\bbD$. The Jost
solution exists under certain additional assumptions (cf.
\cite[formulae (13.9.2)--(13.9.4)]{Simon2}), which are met for
$J\in\cR$. $\phi_0$ is called the Jost function. For an exhaustive
treatment of the Szeg\H{o} and the Jost asymptotics see
\cite{Damsim1}, \cite[Section 13.9]{Simon2}.

The relation between the Szeg\H{o} asymptotics and the Jost function
is given by (see, e.g., \cite{Damsim2}, \cite[Theorem
13.9.2]{Simon2}))
\begin{equation}\label{segjos}
\phi_0(z)=(1-z^2)Q(z)=\frac{B(z,J)}{D(z,J)}\,.
\end{equation}
Hence, for $J\in\cR$ $\phi_0\in H^p$ with $p<\infty$, so the
boundary values of $\phi_0$ exist a.e., and the scattering function
$s$ \eqref{scafun} is well defined. Moreover, by \eqref{outer2}
\begin{equation}\label{scafunr}
\begin{split} s(t,J)=& \frac{(1-t)^{\ga_1(J)}(1+t)^{\ga_2(J)}}{(1-\bar
t)^{\ga_1(J)}(1+\bar t)^{\ga_2(J)}}\,\frac{D_0(\bar
t,J)}{D_0(t,J)}\,B^2(t,J)\\
=& (-1)^{\ga_1(J)}\,t^{\ga_1(J)+\ga_2(J)}\,\frac{D_0(\bar
t,J)}{D_0(t,J)}\,B^2(t,J).
\end{split}
\end{equation}
Clearly, $s(\bar t)=\overline{s(t)}=s^{-1}(t)$.

\begin{theorem}\label{th1}
The scattering function of a Jacobi matrix $J\in\cR$ belongs to
$\Bes$ and admits representation
\begin{equation}\label{scafrep}
s(t,J)=(-1)^{\ga_1(J)}\,t^M e^{-iv(t)},
\end{equation}
where $M=2N+\ga_1(J)+\ga_2(J)\in\bbZ_+$, $v$ satisfies
\begin{equation}\label{besant}
v(t)=\overline{v(t)}=-v(\bar t), \qquad v\in\Bes.
\end{equation}
\end{theorem}
\begin{proof} We apply \eqref{scafunr}. By \eqref{outer1}
$$
\frac{D_0(\bar t,J)}{D_0(t,J)}=\frac{\overline{D_0(t,J)}}{D_0(t,J)}=
\frac{|D_0(t,J)|^2}{D_0^2(t,J)}=e^{-iv_0(t)}. $$
 Since $u_0=\log\wh\rho_0\in\Bes$, and the Hilbert transform is
 bounded (isometric) in $\Bes$, then $v_0\in\Bes$ and antisymmetric.

Next,
$$ B^2(t,J)=t^{2N}\(\prod_{k=1}^N\frac{1-z_k\bar t}{1-z_k
t}\)^2=t^{2N}e^{-iv_1(t)}, $$ and as $z_k\in\bbD$ then $v_1\in\Bes$
and antisymmetric. It remains to put $v=v_0+v_1$. It is clear from
\eqref{scafrep} that $s\in\Bes$, as claimed.
\end{proof}
\begin{remark} Representation \eqref{scafrep} for the scattering
function of $J\in\cR$ is a direct consequence of \eqref{scafunr} and
$\log\wh\rho_0\in\Bes$. A slight refinement of Peller's theorem
\cite[Corollary 7.8.2]{Pel} states that an arbitrary function
$$ h\in\Bes, \qquad h(\bar t)=\overline{h(t)}=h^{-1}(t) $$
a.e. admits representation
$$ h(t)=(-1)^\ga\,t^j e^{-iw(t)}, $$
where $\ga=0$ or $1$, $j\in\bbZ$ an integer number, $w$ satisfies
\eqref{besant}, and such representation is unique. A pair $(\ga, j)$
can be viewed as an index of $h$.
\end{remark}

Let us turn to the numbers $\mu_k(J)$ \eqref{normcon}. In his
version of the scattering theory for Jacobi matrices \eqref{1stmom}
Guseinov suggested the normalizing constants
\begin{equation}\label{guscons}
m_k(J):=\sum_{n=1}^\infty |\phi_n(z_K(J),J)|^2, \quad k=1,2,\ldots,N
\end{equation}
as a part of the scattering data. We show that these values agree.
\begin{proposition}
Let $J\in\cR$. Then $\mu_k(J)=m_k(J)$, $k=1,2,\ldots,N$.
\end{proposition}
\begin{proof} It is known from the general theory of Jacobi matrices
and orthogonal polynomials, that the vectors
$$ \Pi_k=\{s_n(z_k(J))\}_{n\ge1}=\{p_n(\la_k)\}_{n\ge0}\in\ell^2, $$
so $\Pi_k$ are eigenvectors of $J$ with the corresponding
eigenvalues $\la_k$. Furthermore,
\begin{equation}\label{mass}
\frac1{\s_k(J)}=\sum_{n=1}^\infty |s_n(z_k(J))|^2=\sum_{n=0}^\infty
|p_n(\la_k)|^2. \end{equation} On the other hand,
$\phi_0(z_k(J))=0$, and so $\Phi_k=\{\phi_n(z_k(J),J)\}_{n\ge1}$ are
also eigenvectors of $J$ for the same eigenvalues. Hence
$\Phi_k=c_k\Pi_k$, and we find the constants $c_k$ from the initial
data $s_1=1$, so that $c_k=\phi_1(z_k(J),J)$. By \eqref{guscons} and
\eqref{mass} $m_k(J)=|\phi_1(z_k(J),J)|^2\s_k^{-1}(J)$.

It remains to express $\phi_1$ in terms of the spectral data. Once
the Jost asymptotics exists for $J\in\cR$, the Jost solution
$\phi_n$ is proportional to the Weyl solution
$$ w_n(z):=\((z+z^{-1}-J)^{-1}e_1,e_n\), \quad n=1,2,\ldots, \quad
w_0=1, $$ that is, $\phi_n=\phi_0 w_n$. In particular,
$$ \phi_1(z,J)=\phi_0(z,J)w_1(z)=\phi_0(z,J)M(z,J), $$
where $M$ is the Weyl function for $J$
$$ M(z,J)=\((z+z^{-1}-J)^{-1}e_1,e_1\)=\int_{\bbR}
\frac{\s(d\la)}{z+z^{-1}-\la}=\frac{\s_k(J)}{z+z^{-1}-\la_k}+\widetilde
M(z), $$ $\widetilde M$ is analytic at $z_k(J)$. So
$$ \phi_1(z,J)=M(z,J)\frac{B(z,J)}{D(z,J)}. $$
Since $\lim_{z\to z_k}(z-z_k(J))M(z,J)=\s_k(J)(1-z_k^{-2}(J))^{-1}$,
we finally have
$$
\phi_1(z_k(J),J)=\frac{\s_k(J)}{1-z_k^{-2}(J)}\,\frac{B'(z_k(J),J)}{D(z_k(J),J)},
$$
as needed.
\end{proof}

To complete the analysis of scattering data we prove the uniqueness
theorem.
\begin{theorem}\label{uniq}
Let $J_l\in\cR$, $l=1,2$, have the same scattering data. Then
$J_1=J_2$. \end{theorem}
\begin{proof} We want to make sure that $\s(J_1)=\s(J_2)$. It is
clear from \eqref{scafunr} that
$$
\frac{\overline{D_0(t,J_1)}}{D_0(t,J_1)}=\frac{\overline{D_0(t,J_2)}}{D_0(t,J_2)}\,.
$$
But in Ryckman's class $D_0^{\pm1}\in H^2$, so the latter means
$D_0(J_2)=cD_0(J_1)$, $D(J_2)=cD(J_1)$ for some $c>0$, and hence
$\s_{ac}(J_2)=c^2\s_{ac}(J_1)$. Next, $\mu_k(J_1)=\mu_k(J_2)$
implies by \eqref{normcon} $\s_k(J_2)=c^2\s_k(J_1)$, and the
normalizing condition
$$ \int_{-2}^2 f(x,J_1)\,dx + \sum_{k=1}^N \s_k(J_1)=
   \int_{-2}^2 f(x,J_2)\,dx + \sum_{k=1}^N \s_k(J_2)=1 $$
gives $c=1$, as needed.
\end{proof}

\section{Inverse scattering}

Consider the following collection of data $\{\ga_1,\ga_2;\ Z;\
\mu_1,\ldots,\mu_N;\ s\}$:
\begin{enumerate}
    \item a pair of numbers $(\ga_1,\ga_2)$ from
    $\{0,1\}\times\{0,1\}$;
    \item an arbitrary set of $N$ distinct points
    $Z=\{z_k\}_{k=1}^N$ in $(-1,1)\backslash\{0\}$;
    \item an arbitrary set of $N$ positive numbers $\mu_k$;
    \item a function $s\in\Bes$, $|s|=1$ a.e. on $\bbT$, with the index
    $(\ga_1, 2N+\ga_1+\ga_2)$, i.e.,
$$ s(t)=(-1)^{\ga_1}\,t^{2N+\ga_1+\ga_2}\,e^{-i\om(t)}, $$
$\om$ satisfies $\om(t)=\overline{\om(t)}=-\om(\bar t)$,
$\om\in\Bes$.
\end{enumerate}
\begin{theorem}\label{invsca}
There exists a unique Jacobi matrix $J\in\cR$, for which the above
collection is the scattering data.
\end{theorem}
\begin{proof} As in the proof of Theorem \ref{th1} we can write
$$ s(t)= \frac{(1-t)^{\ga_1}(1+t)^{\ga_2}}{(1-\bar
t)^{\ga_1}(1+\bar t)^{\ga_2}}\,B^2(t,Z)e^{-iv_0(t)},
$$
$v_0$ is subject to \eqref{besant}. The Fourier series for $v_0$ is
$$ v_0(t)=\sum_{n\in\bbZ} \wh v_0(n)t^n, \quad \wh
v_0(-n)=\overline{\wh v_0(n)}=-\wh v_0(n), $$
 so $\wh v_0(0)=0$. Take $u_0$ such that $v_0$ is its harmonic
 conjugate. Then $u_0(\bar t)=\overline{u_0(t)}=u_0(t)$ and
$$ u_0(t)=\sum_{n\in\bbZ} \wh u_0(n)t^n, \quad \wh
u_0(-n)=\overline{\wh u_0(n)}=\wh v_0(n). $$
 Note that $u_0$ is defined up to an additive real constant $\wh
 u_0(0)$, which will be chosen later on from the normalization
 condition.

Define a function $\rho_0$ on $[-2,2]$ by $\wh\rho_0=e^{u_0}$, and
put
$$ \rho(x):=\frac{\rho_0(x)}{(2-x)^{\ga_1}(2+x)^{\ga_2}}\,, \quad
 f(x):=\frac1{2\pi}\rho(x)\sqrt{4-x^2}, $$
both up to a factor $C=e^{\wh u_0(0)}$. Next, write
\begin{equation*}
 \begin{split}
D_0(z)=&
\exp\left\{\frac12\,\int_{\bbT}\frac{t+z}{t-z}\,u_0(t)m(dt)\right\}
= \exp\left\{\frac{u_0(z)+iv_0(z)}2\right\}, \\
D(z)=&\frac{D_0(z)}{(1-z)^{\ga_1}(1+z)^{\ga_2}}\,,
\end{split}
\end{equation*}
and put
$$
\s_k:=\mu_k\left|\frac{D(z_k)}{B'(z_k)}\right|^2\,\left|1-z_k^{-2}\right|^{-2}>0,
\quad k=1,2,\ldots,N, $$ the latter values are defined up to a
factor $C$ above, which is now taken from
$$ \int_{-2}^2 f(x)\,dx + \sum_{k=1}^N \s_k=1. $$
Since $v_0\in\Bes$, then so is $u_0$, and by Ryckman's theorem the
measure $\s=\{f,\{\s_k\}\}$ is the spectral measure of some Jacobi
matrix $J\in\cR$. By construction, $\{\ga_1,\ga_2;\ Z;\
\mu_1,\ldots,\mu_N;\ s\}$ is the scattering data for $J$, and $J$ is
unique by Theorem \ref{uniq}. The proof is complete.
\end{proof}

\section{Examples}

As we see, it is comparatively easy to restore the spectral measure
from the scattering data. Assume first that $N=0$, that is,
$\text{supp}\,\s\subset[-2,2]$. To find the Jacobi parameters it
seems reasonable now to carry the measure over from $[-2,2]$ to the
unit circle (the inverse Szeg\H{o} transform)
$\s=\text{Sz}^{(\text{o})}(\mu)$ \footnote{Due to the form of
\eqref{acsp} it is convenient to use the modified Szeg\H{o}
transform}
\begin{equation}\label{sztran}
\s(dx)=f(x)\,dx, \qquad \mu(dt)=w(t)m(dt),
   \quad w(t)=\frac{c\wh f(t)}{|1-t^2|}\,,\end{equation}
$\mu$ is a probability measure on $\bbT$ (by $c$ we denote different
positive normalizing constants), to compute the Verblunsky
coefficients $\a_n=\a_n(\mu)$ of $\mu$, and then go back to Jacobi
parameters with the help of the Geronimus relations
\begin{equation}\label{geron}
\begin{split}
b_{n+1} =& \a_{2n}(1-\a_{2n+1})-\a_{2n+2}(1+\a_{2n+1}), \\
a^2_{n+1} =& (1-\a_{2n+3}(1-\a^2_{2n+2})(1+\a_{2n+1}), \quad
n=0,1,\ldots,
\end{split}
\end{equation}
see \cite[Section 13.2]{Simon2}.

\begin{example}
Let $\ga_1=\ga_2=0$, $a\in[0,1)$, and $s(t)=(1-at)(1-a\bar t)^{-1}$.
Since $B=1$ we have as in the proof of Theorem \ref{invsca}
$$ D(z)=\frac{c}{1-az}\,,\quad
\wh\rho(t)=|D(t)|^2=\frac{c}{|1-at|^2}\,, \quad
f(x)=\frac{c\sqrt{4-x^2}}{1-ax+a^2}\,, $$ so by \eqref{sztran}
$$ \mu(dt)=\frac{cm(dt)}{|1-at|^2}\,. $$
Hence $\mu$ is the Bernstein--Szeg\H{o} measure, for which the
Verblunsky coefficients are
$$ \a_0=a, \quad \a_1=\a_2=\ldots=0. $$
Finally, by \eqref{geron}
$$ b_1=a, \quad b_2=b_3=\ldots=0, \qquad a_1=a_2=\ldots=1. $$
\end{example}

\begin{example}
For $\ga_1, \ga_2, a$ as above put $s(t)=(1-a\bar t)(1-at)^{-1}$. We
now have
$$ D(z)=c(1-az)\,,\quad
\wh\rho(t)=c|1-at|^2\,, \quad f(x)=c(1-ax+a^2)\sqrt{4-x^2},
$$
and $\mu(dt)=c|1-at|^2\,m(dt)$. The Verblunsky coefficients are (see
\cite[Example 1.6.4]{Simon1})
$$ \a_n=-\frac{a^{-1}-a}{a^{-n-2}-a^{n+2}}\,, \qquad n=0,1,\ldots.
$$
Finally,
$$ b_{n+1}=-a^{2n+1}\frac{(1-a^2)^2}{(1-a^{2n+2})(1-a^{2n+4})}\,,
\quad a^2_{n+1}=1-a^{2n+2}\frac{(1-a^2)^2}{(1-a^{2n+4})^2}\,. $$
\end{example}

It is worth pointing at the difference between the above examples.
In the first one $D^{-1}$ is a polynomial (of degree 1), so there
are finitely many nonzero Verblunsky coefficients $\a_n$, and $J$ is
of finite support, i.e., $a_n=1$ and $b_n=0$ for all large enough
$n$. In the second case $D^{-1}$ is a rational function with the
pole $1/a$, so the Jacobi parameters tend to their limits
exponentially fast.

\begin{example}\label{ex3}
Let $\ga_1=\ga_2=N=0$, $a,b\in[0,1)$. Put
$$ s(t)=\frac{(1-at)(1-bt)}{(1-a\bar t)(1-b\bar t)}\,. $$
As above, the spectral measure is
$$ \s(dx)=\frac{c\sqrt{4-x^2}}{(1-ax+a^2)(1-bx+b^2)}\,dx,\qquad
\mu(dt)=\frac{cm(dt)}{|1-at|^2|1-bt|^2}\,. $$ The latter is again
the Bernstein--Szeg\H{o} measure with the Verblunsky coefficients
$$ \a_0=\frac{a+b}{1+ab}\,, \ \ \a_1=-ab, \ \ \a_2=\a_3=\ldots=0. $$
Finally,
$$ b_1=a+b, \ \ b_2=b_3=\ldots=0; \quad a^2_1=1-ab,\ \
a_2=a_3=\ldots=1. $$
\end{example}

\begin{remark} For $\ga_1=\ga_2=0$ the Szeg\H{o} transform
$\text{Sz}^{(\text{o})}$ is a right one in the sense of
\cite{Ryck2}. We can compute the same examples with
$\ga_1,\ga_2=0,1$ by using all four Szeg\H{o} transforms
$\text{Sz}^{(\text{o})}$, $\text{Sz}^{(\text{e})}$,
$\text{Sz}^{(\pm)}$, correspondingly.
\end{remark}

To complete the examples section consider the case $N=1$, that is,
the spectral measure has a mass point outside $[-2,2]$. Now the
Szeg\H{o} transforms do not work, so we proceed in two steps. First,
we compute the Jacobi parameters $a_n, b_n$ for $\s_1=c\s_{ac}$
(properly normalized), and then we add a mass point and recompute
the Jacobi parameters (cf., e.g., \cite[Lemma 7.15]{Nev79}).
\begin{example}
Let $\ga_1=\ga_2=0$, $N=1$, $z_1\in (0,1)$ and $\mu_1>0$ are given,
and $s(t)=t^2$. We have
$$ B^2(t)=\(\frac{t-z_1}{1-z_1t}\)^2=t^2\(\frac{1-z_1\bar t}{1-z_1t}\)^2, $$
so
$$ s(t)=B^2(t)\frac{D(\bar t)}{D(t)}\,, \qquad
D(z)=\frac{c_0}{(1-z_1z)^2}\,.
$$ The spectral
measure is of the form
$$ \s(dx)=\frac{c_0^2\sqrt{4-x^2}}{2\pi(1-z_1x+z_1^2)^2}\,dx+\s_1\,\de(\la_1),
\quad \la_1=z_1+\frac1{z_1} $$
with
$$ \s_1=\mu_1\,\left|\frac{D(z_1)}{B'(z_1)}\right|^2\,|1-z_1^{-2}|^2=
c_0^2\frac{\mu_1}{z_1^4}\,. $$
$c_0$ is determined from $\s(\bbR)=1$, or
$$ c_0^{-2}=\int_{-2}^2 f_0(x)dx+\frac{\mu_1}{z_1^4},
\qquad f_0(x)=\frac{\sqrt{4-x^2}}{2\pi(1-z_1x+z_1^2)^2}\,. $$

Take $\s_0(dx)=c_1^2f_0dx$, the measure on $[-2,2]$, $\s_0(\bbR)=1$.
It is not hard to find the values of normalizing constants $c_0,
c_1$:
$$ c_0^{-2}=\frac1{1-z_1^2}+\frac{\mu_1}{\z_1^4}\,, \quad
c_1^{-2}=\int_{-2}^2\frac{\sqrt{4-x^2}}{2\pi(1-z_1x+z_1^2)^2}\,dx=\frac1{\sqrt{1-z_1^2}}\,.
$$ If
$\mu_1(dt)=w_1(t)m(dt)$ and $\s_0=\text{Sz}^{(\text{o})}(\mu_1)$,
then $w_1=c|1-z_1t|^{-4}$, and as in Example \ref{ex3} (with
$a=b=z_1$) \begin{equation}\label{par} b_1(\s_0)=2z_1, \ \
b_2(\s_0)=\ldots=0; \quad a^2_1(\s_0)=1-z_1^2,\ \
a_2(\s_0)=\ldots=1.
\end{equation} Note that $\s_0$ is the Bernstein--Szeg\H{o} measure
on $[-2,2]$ (see \cite[Section II.2.6]{Sze}), and the corresponding
orthonormal polynomials are known explicitly
$$
p_n(\s_0,z+z^{-1})=\frac{z^{n+1}\(1-\frac{z_1}{z}\)^2-z^{-n-1}(1-z_1z)^2}
{\sqrt{1-z_1^2}\,(z-z^{-1})}\,, \quad n=1,2,\ldots, \quad p_0=1. $$
In particular, $p_k(\s_0,\la_1)=\sqrt{1-z_1^2}\,z_1^{-k}$,
$k=1,2,\ldots$. The Christoffel kernels are
$$ K_{n+1}(\s_0,\la_1)=\sum_{k=0}^n
p^2_k(\s_0,\la_1)=z_1^{-2n}, \quad n=0,1,\ldots,\quad K_0=0.
$$

To apply Nevai's formulae write the perturbed measure $\s$ in a
canonical form
\begin{equation}\label{canon}
\s(dx)=\frac{\s_0(dx)+\ep\de(\la_1)}{1+\ep}\,, \quad
\ep=\frac{c_1^2}{c_0^2}-1=\frac{\mu_1}{z_1^4}(1-z_1^2),
\end{equation} so
\begin{equation*}
\begin{split}
 a^2_n(\s)= & a^2_n(\s_0)\,\frac{(1+\ep K_{n-1}(\s_0,\la_1))(1+\ep
K_{n+1}(\s_0,\la_1))}{(1+\ep K_{n}(\s_0,\la_1))^2}\,, \\
b_{n}(\s) =& b_{n}(\s_0)-a_{n-1}(\s_0)V_{n-1}+a_{n}(\s_0)V_{n},
\quad n=1,2,\ldots,
\end{split}
\end{equation*}
where
$$
V_n = \frac{\ep p_{n-1}(\s_0,\la_1)p_n(\s_0,\la_1)}{1+\ep
K_{n}(\s_0,\la_1)}=\left\{%
\begin{array}{ll}
    \frac{\ep(1-z_1^2)}{z_1(z_1^{2n-2}+\ep)}, & \hbox{$n\ge2$;} \\
    \frac{\ep\,\sqrt{1-z_1^2}}{z_1(1+\ep)}, & \hbox{$n=1$.} \\
\end{array}%
\right.,\quad V_0=0. $$ Eventually, we have
\begin{equation*}
\begin{split}
a^2_1(\s)=&(1-z_1^2)^2\,\frac{1+\mu_1(1+z_1^{-2})}{(1+\mu_1z_1^{-4}(1-z_1^2))^2}\,,\\
a^2_n(\s) =& 1+\frac{\ep(1-z_1^2)^2}{(z_1^{2n}+\ep
z_1^2)^2}\,z_1^{2n}, \quad n=2,3,\ldots,
\end{split}
\end{equation*}
and
$$ b_n(\s)=\frac{\ep(1-z_1^2)^2}{(\ep+z_1^{2n-2})(\ep+
z_1^{2n-4})}\,z_1^{2n-5}, \qquad n\ge3, $$ $\ep$ from \eqref{canon}.
\end{example}
\begin{remark} The Jost function is now
$$ \phi_0(z)=c(1-z_1z)\(1-\frac{z}{z_1}\), $$
but $J$ is of infinite support. The Jacobi parameters tend to their
limits exponentially fast, cf. \cite[Remark 1.10]{Damsim2}.
\end{remark}

\bibliographystyle{amsplain}

\end{document}